\title{Waiter-Client and Client-Waiter Hamiltonicity games on random graphs\footnote{\copyright 2017. This manuscript version is made available under the CC-BY-NC-ND 4.0 license http://creativecommons.org/licenses/by-nc-nd/4.0/}}

\author{
Dan Hefetz
\thanks{Department of Computer Science, Ariel University, Ariel 40700, Israel. Email: danhe@ariel.ac.il
} 
\and Michael
Krivelevich \thanks{School of Mathematical Sciences, Raymond and
Beverly Sackler Faculty of Exact Sciences, Tel Aviv University,
6997801, Israel. Email: krivelev@post.tau.ac.il. Research supported in
part by USA-Israel BSF Grant 2014361 and by grant 912/12 from the
Israel Science Foundation.}
\and Wei En Tan
\thanks{School of Mathematics, University of Birmingham, Edgbaston, Birmingham B15 2TT, United
Kingdom. Email: WET916@bham.ac.uk.} 
}

\documentclass[11pt]{article}
\usepackage{amsmath,amsthm,amssymb,latexsym,color,epsfig,a4,enumerate}

\newtheorem{theorem}{Theorem} [section]
\newtheorem{lemma}[theorem]{Lemma}
\newtheorem{proposition}[theorem]{Proposition}
\newtheorem{remark}[theorem]{Remark}
\newtheorem{definition}[theorem]{Definition}

\newtheorem{corollary}[theorem]{Corollary}

\begin{document}

\maketitle

\begin{abstract}
We study two types of two player, perfect information games with no chance moves, played on the edge set of the binomial random graph ${\mathcal G}(n,p)$. In each round of the $(1 : q)$ Waiter-Client Hamiltonicity game, the first player, called Waiter, offers the second player, called Client, $q+1$ edges of ${\mathcal G}(n,p)$ which have not been offered previously. Client then chooses one of these edges, which he claims, and the remaining $q$ edges go back to Waiter. Waiter wins this game if by the time every edge of ${\mathcal G}(n,p)$ has been claimed by some player, the graph consisting of Client's edges is Hamiltonian; otherwise Client is the winner. Client-Waiter games are defined analogously, the main difference being that Client wins the game if his graph is Hamiltonian and Waiter wins otherwise. In this paper we determine a sharp threshold for both games. Namely, for every fixed positive integer $q$, we prove that the smallest edge probability $p$ for which a.a.s. Waiter has a winning strategy for the $(1 : q)$ Waiter-Client Hamiltonicity game is $(1 + o(1)) \log n/n$, and the smallest $p$ for which a.a.s. Client has a winning strategy for the $(1 : q)$ Client-Waiter Hamiltonicity game is $(q + 1 + o(1)) \log n/n$.  
\end{abstract}

\section{Introduction} \label{intro}

The theory of positional games on graphs and hypergraphs goes back to the seminal papers of Hales and Jewett~\cite{HJ}, Lehman~\cite{Lehman}, and Erd\H{o}s and Selfridge~\cite{ES}. It has since become a highly developed area of combinatorics (see, e.g., the monograph of Beck~\cite{TTT} and the recent monograph~\cite{HKSSbook}). The most popular and widely studied positional games are the so-called Maker-Breaker games. Let $q$ be a positive integer, let $X$ be a finite set and let ${\mathcal F}$ be a family of subsets of $X$. The set $X$ is the \emph{board} of the game and the elements of ${\mathcal F}$ are the \emph{winning sets}. In each round of the biased $(1 : q)$ Maker-Breaker game $(X, {\mathcal F})$, Maker claims one previously unclaimed element of $X$ and then Breaker responds by claiming $q$ previously unclaimed elements of $X$. Maker wins this game if, by the time every element of $X$ has been claimed, he has claimed all elements of some set $A \in {\mathcal F}$; otherwise Breaker is the winner. Since this is a finite, perfect information game with no chance moves and no possibility of a draw, one of the two players must have a winning strategy.  

The so-called Avoider-Enforcer games form another class of well-studied positional games. In these games, Enforcer aims to force Avoider to claim all elements of some set $A \in {\mathcal F}$. Avoider-Enforcer games are sometimes referred to as mis\`ere Maker-Breaker games. There are two different sets of rules for Avoider-Enforcer games: \emph{strict rules} under which the number of board elements a player claims per round is precisely his bias and \emph{monotone rules} under which the number of board elements a player claims per round is at least as large as his bias (for more information on Avoider-Enforcer games see, for example,~\cite{HKSae, HKSSae, HKSSbook}).

In this paper, we study \emph{Waiter-Client} and \emph{Client-Waiter} positional games. In every round of the biased $(1 : q)$ Waiter-Client game $(X, {\mathcal F})$, the first player, called Waiter, offers the second player, called Client, $q+1$ previously unclaimed elements of $X$. Client then chooses one of these elements which he claims, and the remaining $q$ elements are claimed by Waiter (if, in the final round of the game, strictly less than $q+1$ unclaimed elements remain, then all of them are claimed by Waiter). The game ends as soon as all elements of $X$ have been claimed. Waiter wins this game if he manages to force Client to claim all elements of some $A \in {\mathcal F}$; otherwise Client is the winner. Client-Waiter games are defined analogously. The only two differences are that Client wins this game if and only if he manages to claim all elements of some $A \in {\mathcal F}$ (otherwise Waiter is the winner), and that Waiter is allowed to offer strictly less than $q+1$ (but at least $1$) board elements per round (this is a technical issue which is needed in order to overcome a certain lack of monotonicity; more details can be found in~\cite{Bednarska2013}). Waiter-Client and Client-Waiter games were first defined and studied by Beck under the names Picker-Chooser and Chooser-Picker, respectively (see, e.g.,~\cite{becksec}). 

The interest in Waiter-Client and Client-Waiter games is three-fold. Firstly, they are interesting in their own right. For example, a Waiter-Client (respectively, Client-Waiter) game in which Waiter plays randomly is the well-known avoiding (respectively, embracing) Achlioptas process (without replacement). Many randomly played Waiter-Client and Client-Waiter games have been considered in the literature, often under different names (see, e.g.,~\cite{BF, KLS, KLuS, KSS, MST}). Secondly, they exhibit a strong probabilistic intuition (see, e.g.,~\cite{becksec, TTT, ctcs, BHL, BHKL, HKT}). That is, the outcome of many natural Waiter-Client and Client-Waiter games is often roughly the same as it would be had both players played randomly (although, typically, a random strategy for any single player is very far from optimal). 
Lastly, it turns out that, in some cases, these games are useful in the analysis of Maker-Breaker games (examples and other related issues can be found, e.g., in \cite{becksec, cdm, Bednarska2013, Knox}).

Our focus in this paper is on Waiter-Client and Client-Waiter Hamiltonicity games played on the edge set of the binomial random graph $\mathcal{G}(n,p)$. A \emph{Hamilton cycle} of a graph $G$ is a cycle which passes through every vertex of $G$ exactly once. A graph is said to be \emph{Hamiltonian} if it admits a Hamilton cycle. Hamiltonicity is one of the most central notions in graph theory, and has been intensively studied by numerous researchers for many
years. Let ${\mathcal H} = {\mathcal H}(n)$ denote the property of an $n$-vertex graph being Hamiltonian, that is, ${\mathcal H} = \{G : V(G) = [n] \text{ and } G \text{ is a Hamiltonian graph}\}$. For every positive integer $q$, let ${\mathcal W}_{\mathcal H}^q$ denote the graph property of being Waiter's win in the $(1 : q)$ Waiter-Client Hamiltonicity game on $E(G)$, that is, $G \in {\mathcal W}_{\mathcal H}^q$ if and only if Waiter has a winning strategy for the $(1 : q)$ Waiter-Client game $(E(G), {\mathcal H})$. Similarly, let ${\mathcal C}_{\mathcal H}^q$ denote the graph property of being Client's win in the $(1 : q)$ Client-Waiter Hamiltonicity game on $E(G)$, that is, $G \in {\mathcal C}_{\mathcal H}^q$ if and only if Client has a winning strategy for the $(1 : q)$ Client-Waiter game $(E(G), {\mathcal H})$. Note that, if Waiter has a winning strategy for the $(1:q)$ Waiter--Client game $(E(G),\mathcal{H})$ for some graph $G$, he may use this same strategy to also win the $(1:q)$ Waiter--Client game $(E(G'),\mathcal{H})$ for any graph $G'$ that contains $G$ as a subgraph and for which $V(G')=V(G)$. Hence, graph property ${\mathcal W}_{\mathcal H}^q$ is monotone increasing. By similar reasoning, this is also true for ${\mathcal C}_{\mathcal H}^q$. 

We are interested in finding the minimum density that a graph typically needs to ensure that it satisfies the property ${\mathcal W}_{\mathcal H}^q$ (and similarly ${\mathcal C}_{\mathcal H}^q$). Formally, we would like to find sharp thresholds for ${\mathcal W}_{\mathcal H}^q$ and for ${\mathcal C}_{\mathcal H}^q$, whose existence is guaranteed due to the monotonicity of these properties (see \cite{Bollobas1986,Friedgut1999}). That is, for every fixed positive integer $q$, we would like to find functions $p_{\mathcal{W}, q} = p_{\mathcal{W}, q}(n)$ and $p_{\mathcal{C}, q} = p_{\mathcal{C}, q}(n)$ such that, for every $\varepsilon > 0$, we have 
$$
\lim_{n \rightarrow \infty} \mathbb{P}[{\mathcal G}(n, (1 - \varepsilon) p_{\mathcal{W}, q}) \in {\mathcal W}_{\mathcal H}^q] = 0 \,\,\,\,\, \text{and} \,\,\,\,
\lim_{n \rightarrow \infty} \mathbb{P}[{\mathcal G}(n, (1 + \varepsilon) p_{\mathcal{W}, q}) \in {\mathcal W}_{\mathcal H}^q] = 1,
$$       
and similarly,
$$
\lim_{n \rightarrow \infty} \mathbb{P}[{\mathcal G}(n, (1 - \varepsilon) p_{\mathcal{C}, q}) \in {\mathcal C}_{\mathcal H}^q] = 0 \,\,\,\,\, \text{and} \,\,\,\,
\lim_{n \rightarrow \infty} \mathbb{P}[{\mathcal G}(n, (1 + \varepsilon) p_{\mathcal{C}, q}) \in {\mathcal C}_{\mathcal H}^q] = 1.
$$

A classical result in the theory of random graphs, due to Koml\'os and Szemer\'edi~\cite{KSz} and independently Bollob\'as~\cite{BolHam}, asserts that $\log n/n$ is a sharp threshold for the appearance of a Hamilton cycle in ${\mathcal G}(n,p)$. Our first result shows that the same function is also a sharp threshold for the property ${\mathcal W}_{\mathcal H}^q$, for every fixed positive integer $q$. 

\begin{theorem} \label{WCHamilton}
Let $q$ be a positive integer. Then $\log n/n$ is a sharp threshold for the property ${\mathcal W}_{\mathcal H}^q$.
\end{theorem}

Our second result shows that, in contrast to our first result, the sharp threshold for the property ${\mathcal C}_{\mathcal H}^q$ grows with $q$, and even for $q=1$, is already larger than the threshold for the Hamiltonicity of ${\mathcal G}(n,p)$.  

\begin{theorem} \label{CWHamilton}
Let $q$ be a positive integer. Then $(q+1) \log n/n$ is a sharp threshold for the property ${\mathcal C}_{\mathcal H}^q$.
\end{theorem}

A trivial necessary condition for Hamiltonicity is minimum degree at least $2$. The latter is not achieved if $p \leqslant (1 - o(1)) \log n/n$. On the other hand, for $p \geqslant (1 + o(1)) \log n/n$, it is not hard to show that a.a.s. (\emph{i.e.}, with probability tending to 1 as $n$ tends to infinity) Waiter can force Client to build a graph with large minimum degree. This partly explains the location of the threshold in Theorem~\ref{WCHamilton}. In Theorem~\ref{CWHamilton} we need to consider the degree sequence of ${\mathcal G}(n,p)$ more carefully. In order to win the Client-Waiter Hamiltonicity game, it is enough for Waiter to find \emph{many} vertices of \emph{small} degree and then isolate one of them in Client's graph. On the other hand, if all degrees are sufficiently large (enough for the sum $\sum_{v \in V(G)} \left(\frac{q}{q+1} \right)^{d_G(v)}$ to be very small, where $G \sim {\mathcal G}(n,p)$), then Client can build a graph with large minimum degree. Balancing these two properties of the degree sequence will determine the location of the threshold in Theorem~\ref{CWHamilton}.        

Theorems~\ref{WCHamilton} and~\ref{CWHamilton} determine sharp thresholds for $(1 : q)$ Waiter-Client and Client-Waiter Hamiltonicity games for every fixed positive integer $q$. Somewhat surprisingly, the best known analogous results for $(1 : q)$ Maker-Breaker and Avoider-Enforcer Hamiltonicity games on ${\mathcal G}(n,p)$, where $q \geqslant 2$ is a fixed integer, are not as accurate. It was conjectured in~\cite{SS} that, for every $1 \leqslant q \leqslant (1 - o(1)) n/\log n$, the smallest edge probability $p$ for which a.a.s. Maker has a winning strategy in the $(1 : q)$ Maker-Breaker Hamiltonicity game is $\Theta(q \log n/n)$. This was proved in~\cite{FGKN}, where an analogous statement for Avoider-Enforcer games was proved as well. An even stronger result was proved in~\cite{FGKN} under the additional assumption that $q = \omega(1)$. In this case, the graph property of being Maker's win in the $(1 : q)$ Maker-Breaker Hamiltonicity game has a sharp threshold at $q \log n/n$. Very accurate results are known for the $(1 : 1)$ Maker-Breaker Hamiltonicity game on a random graph (see~\cite{HKSSham, BFHK}).

\section{Preliminaries}

\noindent For the sake of simplicity and clarity of presentation, we do not make a par\-ti\-cu\-lar effort to optimize the constants obtained in some of our proofs. We also omit floor and ceiling signs whenever these are not crucial. Most of our results are asymptotic in nature and whenever necessary we assume that the number of vertices $n$ is sufficiently large. Throughout this paper, $\log$ stands for the natural logarithm, unless explicitly stated otherwise. Our graph-theoretic notation is standard and follows that of~\cite{West}. In particular, we use the following.

For a graph $G$, let $V(G)$ and $E(G)$ denote its sets of vertices and edges respectively, and let $v(G) = |V(G)|$ and $e(G) = |E(G)|$. For a set $A \subseteq V(G)$, let $E_G(A)$ denote the set of edges of $G$ with both endpoints in $A$ and let $e_G(A) = |E_G(A)|$. For disjoint sets $A,B \subseteq V(G)$, let $E_G(A,B)$ denote the set of edges of $G$ with one endpoint in $A$ and one endpoint in $B$, and let $e_G(A,B) = |E_G(A,B)|$. For a set $S \subseteq V(G)$, let $G[S]$ denote the subgraph of $G$ induced on the set $S$. For a set $S \subseteq V(G)$, let $N_G(S) = \{v \in V(G) \setminus S : \exists u \in S \text{ such that } uv \in E(G)\}$ denote the \emph{outer neighbourhood} of $S$ in $G$. For a vertex $u \in V(G)$ we abbreviate $N_G(\{u\})$ under $N_G(u)$ and let $d_G(u) = |N_G(u)|$ denote the \emph{degree} of $u$ in $G$. The \emph{maximum degree} of a graph $G$ is $\Delta(G) = \max \{d_G(u) : u \in V(G)\}$ and the \emph{minimum degree} of a graph $G$ is $\delta(G) = \min \{d_G(u) : u \in V(G)\}$. Often, when there is no risk of confusion, we omit the subscript $G$ from the notation above. 

For any family $\mathcal{F}$ of subsets of some set $X$, we define the \emph{transversal} of $\mathcal{F}$ to be the set $\mathcal{F}^* = \{A \subseteq X : A \cap B \neq \emptyset \text{ for every } B \in \mathcal{F}\}$. 

Assume that some Waiter-Client or Client-Waiter game, played on the edge-set of some graph $H = (V,E)$, is in progress. At any given moment during this game, let $E_W$ denote the set of all edges that were claimed by Waiter up to that moment, let $E_C$ denote the set of all edges that were claimed by Client up to that moment, let $G_W = (V, E_W)$ and let $G_C = (V, E_C)$. Moreover, let $G_F = (V, E_F)$, where $E_F = E \setminus (E_W \cup E_C)$; the edges of $E_F$ are called \emph{free}. 

Throughout the paper we will use the following well-known concentration inequalities (see, e.g.,~\cite{Alon2008}).

\begin{theorem}[Chernoff] \label{Chernoff}
If $X \sim \text{Bin}(n,p)$, then
\begin{itemize}
\item[(i)] $\mathbb{P}[X < (1-a) n p] < \exp\left(-\frac{a^2 n p}{2}\right)$ for every $a > 0$.
\item[(ii)] $\mathbb{P}[X > (1+a) n p] < \exp\left(-\frac{a^2 n p}{3}\right)$ for every $0 < a < 1$.
\end{itemize}
\end{theorem}

\begin{theorem}[Chebyshev] \label{Chebyshev}
If $X$ is a random variable with $\mathbb{E}[X] < \infty$ and $\text{Var}[X] < \infty$, then for any $k > 0$
$$
\mathbb{P}[|X - \mathbb{E}[X]| \geqslant k] \leqslant \frac{\text{Var}[X]}{k^2}.
$$
\end{theorem}

The rest of this paper is organized as follows: In Section~\ref{sec::gameTools} we state and prove various results about Waiter-Client and Client-Waiter games, some of which have independent interest. In Section~\ref{sec::Gnp} we explore several properties of random graphs. In Section~\ref{sec::expanders} we discuss the relation between expanders and Hamiltonicity. Using the results derived in the previous three sections, we prove Theorem~\ref{WCHamilton} in Section~\ref{sec::mainWC} and Theorem~\ref{CWHamilton} in Section~\ref{sec::mainCW}. Finally, in Section~\ref{sec::openprob} we discuss possible directions for future research.    

\section{Game-Theoretic Tools} \label{sec::gameTools}

In this section we state and prove various winning criteria for Waiter-Client and Client-Waiter games which we will use in our proofs of Theorems~\ref{WCHamilton} and~\ref{CWHamilton}. We begin by stating a result of Bednarska-Bzd\c ega~\cite{Bednarska2013} which provides a winning criteria for Waiter in biased Waiter-Client transversal games.

\begin{theorem}[\cite{Bednarska2013}] \label{Bednarska}
Let $q$ be a positive integer, let $X$ be a finite set and let ${\mathcal F}$ be a family of subsets of $X$. If
$$
\sum_{A \in {\mathcal F}} 2^{-|A|/(2q-1)} < 1/2,
$$
then Waiter has a winning strategy for the $(1 : q)$ Waiter-Client game $(X, {\mathcal F}^*)$.
\end{theorem}

Next, we state and prove a sufficient condition for Client to win biased Client-Waiter transversal games.

\begin{theorem} \label{ClientWinningTransversalGame}
Let $q$ be a positive integer, let $X$ be a finite set and let $\mathcal{F}$ be a family of subsets of $X$. If
$$
\sum_{A \in \mathcal{F}} \left(\frac{q}{q+1}\right)^{|A|} < 1,
$$
then Client has a winning strategy for the $(1:q)$ Client-Waiter game $(X, \mathcal{F}^*)$.
\end{theorem}

\begin{remark}
It is not hard to adapt Beck's winning criterion for Breaker in biased Maker-Breaker games (see \cite{TTT}) to prove that, if $\sum_{A \in \mathcal{F}} 2^{- |A|/q} < 1$, then Client has a winning strategy for the $(1:q)$ Client-Waiter game $(X, \mathcal{F}^*)$. However, note that Theorem~\ref{ClientWinningTransversalGame} does not provide a weaker result since $q/(q+1)\leqslant 2^{-1/q}$ for every $q\geqslant 1$, with equality if and only if $q=1$.
\end{remark}

\noindent \emph{Proof of Theorem~\ref{ClientWinningTransversalGame}}.
Client will play randomly, that is, in each round he will choose one of the elements Waiter offers him uniformly at random, independently of all previous choices. Since Client-Waiter games are finite, perfect information games with no chance moves and no draws, in order to prove that Client has a winning strategy, it suffices to show that, given any fixed strategy ${\mathcal S}_W$ of Waiter, 
$$
\mathbb{P}[\text{Client loses }(X, \mathcal{F}^*) \,\, | \text{ Waiter follows } {\mathcal S}_W] < 1.
$$  

Fix some strategy ${\mathcal S}_W$ of Waiter and a set $A \in \mathcal{F}^*$. Given that Waiter plays according to ${\mathcal S}_W$, let $r$ denote the total number of rounds played in the game and, for every $1 \leqslant i \leqslant r$, let $Z_i$ denote the set of elements Waiter offers Client in the $i$th round, let $z_i = |Z_i|$ and let $a_i = |A \cap Z_i|$. Note that $r$, $z_i$ and $a_i$ might depend on Client's random choices. For every $1 \leqslant i \leqslant r$, given $z_i$ and $a_i$, the probability that Client claims an element of $A$ in the $i$th round is $a_i/z_i$, independently of his previous choices. Hence, the probability that Client does not claim any element of $A$ throughout the game is 
$$
\prod_{i=1}^r \left(1 - \frac{a_i}{z_i}\right) \leqslant \prod_{i=1}^r \left(1 - \frac{a_i}{q+1}\right) \leqslant \prod_{i=1}^r \left(1 - \frac{1}{q+1}\right)^{a_i} = \left(\frac{q}{q+1}\right)^{|A|} \,,
$$
where the second inequality holds by Bernoulli's inequality.       

Taking a union bound over the elements of $\mathcal{F}$, we conclude that
$$
\mathbb{P}[\text{Client loses }(X,\mathcal{F}^*) \,\, | \text{ Waiter follows } {\mathcal S}_W] \leqslant \sum_{A \in \mathcal{F}} \left(\frac{q}{q+1}\right)^{|A|} < 1 \,,
$$
as claimed.
{\hfill $\Box$ \medskip\\}

A simple step in Waiter's strategy to force Client to build a Hamilton cycle, is to force him to quickly build a graph with large minimum degree. Our next result shows that this is indeed possible.

\begin{lemma} \label{lem::largeMinDeg}
Let $G$ be a graph on $n$ vertices with minimum degree $\delta(G)=\delta$ and let $q$ and $\gamma \leqslant \left\lfloor \frac{\delta}{2(q+1)} \right\rfloor$ be positive integers. When playing a $(1 : q)$ Waiter-Client game on $E(G)$, Waiter has a strategy to force Client to build a spanning subgraph of $G$ with minimum degree at least $\gamma$, by offering at most $(q+1) \gamma n$ edges of $G$.
\end{lemma}

\begin{proof}
Let $u_1, \ldots, u_n$ denote the vertices of $G$. We define a new graph $G^*$, where $G^* = G$ if $d_G(u_i)$ is even for every $1 \leqslant i \leqslant n$, and otherwise $G^*$ is the graph obtained from $G$ by adding a new vertex $v^*$ and connecting it to every vertex of odd degree in $G$. Since all degrees of $G^*$ are even, it admits an Eulerian orientation $\overrightarrow{G^*}$. For every $1 \leqslant i \leqslant n$, let $E(u_i) = \{u_i u_j \in E(G) : u_i u_j \text{ is directed from } u_i \text{ to } u_j \text{ in } \overrightarrow{G^*}\}$. It is evident that $|E(u_i)| \geqslant \lfloor \delta/2 \rfloor \geqslant (q+1) \gamma$ for every $1 \leqslant i \leqslant n$ and that the sets $E(u_1), \ldots, E(u_n)$ are pairwise disjoint. 

For every $1 \leqslant i \leqslant n$ and every $1 \leqslant j \leqslant \gamma$, in the $\left((i-1) \gamma + j \right)$th round of the game, Waiter offers Client $q+1$ arbitrary free edges of $E(u_i)$. It is evident that, after offering at most $(q+1) \gamma$ edges of $E(u_i)$ for every $1 \leqslant i \leqslant n$, the minimum degree of Client's graph is at least $\gamma$.    
\end{proof}

The rest of this section is devoted to a Client-Waiter version of the so-called \emph{Box Game}. The Maker-Breaker version of this game was introduced by Chv\'atal and Erd\H{o}s in their seminal paper~\cite{CE} and was subsequently fully analyzed by Hamidoune and Las Vergnas in~\cite{HL}. The version of the box game we are interested in, which we will refer to as the $(1 : q)$ Client-Waiter box game, is defined as follows. Let ${\mathcal F} = \{A_1, \ldots, A_n\}$ be a family of pairwise disjoint sets such that $t-1 \leqslant |A_1| \leqslant \ldots \leqslant |A_n| = t$. We refer to such a family as being \emph{canonical} of type $t$. The box game on ${\mathcal F}$ is simply the $(1 : q)$ Client-Waiter game $(\bigcup_{i=1}^n A_i, {\mathcal F}^*)$. Note that Waiter wins the $(1 : q)$ Client-Waiter box game on ${\mathcal F}$ if and only if he is able to claim all elements of some $A_i$. 

Suppose that at some point during the box game on ${\mathcal F}$, Client claims an element of $A_i$ for some $1 \leqslant i \leqslant n$. Since Waiter can no longer claim all elements of $A_i$, neither player has any incentive to claim more elements of $A_i$. Therefore, we can pretend that $A_i$ was removed from ${\mathcal F}$. If on the other hand, Waiter claims an element $a \in A_i$, then we can pretend that instead of trying to fully claim $A_i$, his goal is now to fully claim $A_i \setminus \{a\}$. Hence, we can view the family ${\mathcal F}$, on which the game is played, as changing throughout the game as follows. Assume that ${\mathcal F}_i$ denotes the (multi) family representing the game immediately before the $i$th round; in particular ${\mathcal F}_1 = {\mathcal F}$. Let $W_i$ denote the set of elements Waiter offers Client in the $i$th round, let $c_i \in W_i$ denote the element claimed by Client and let $j$ denote the unique integer for which $c_i \in A_j$. Then we define ${\mathcal F}_{i+1} = \{A \setminus W_i : A \in {\mathcal F}_i \text{ and } A \neq A_j\}$. Using this point of view, we see that Waiter wins the $(1 : q)$ Client-Waiter box game on ${\mathcal F}$ if and only if $\emptyset \in {\mathcal F}_i$ for some positive integer $i$.             

\begin{proposition} \label{BoxGame}
Let $q$ and $t$ be positive integers and let ${\mathcal F}$ be a canonical family of type $t$. If $|{\mathcal F}| \geqslant 2(q+1)^{t+1}/q^t$, then Waiter has a winning strategy for the $(1 : q)$ Client-Waiter box game on ${\mathcal F}$.
\end{proposition}

\begin{remark}
In light of Theorem~\ref{ClientWinningTransversalGame}, Proposition~\ref{BoxGame} is not far from being best possible.
\end{remark}

\noindent \emph{Proof of Proposition~\ref{BoxGame}}.
Waiter plays so as to keep the families ${\mathcal F}_i$ canonical; this is achieved as follows. For every positive integer $i$, let $t_i = \max \{|A| : A \in {\mathcal F}_i\}$, let ${\mathcal L}_i = \{A \in {\mathcal F}_i : |A| = t_i\}$ and let $\ell_i = |{\mathcal L}_i|$. In the $i$th round, Waiter offers Client an arbitrary set $W_i \subseteq \bigcup_{A \in {\mathcal L}_i} A$ of size $\min \{q+1, \ell_i\}$ such that $|A \cap W_i| \leqslant 1$ for every $A \in {\mathcal L}_i$. We claim that this is a winning strategy for Waiter.

For every $0 \leqslant j \leqslant t$, let $i_j$ denote the smallest integer such that ${\mathcal F}_{i_j}$ is canonical of type $j$ (to make this well-defined, we view the empty family as being canonical of every type). In particular, $i_t = 1$ and, in order to prove our claim, it suffices to show that $|{\mathcal F}_{i_0}| \geqslant 1$. We will in fact prove a more general claim, namely, that
\begin{equation} \label{eq4}
|{\mathcal F}_{i_j}| \geqslant \left(\frac{q}{q+1}\right)^{t-j} |{\mathcal F}| - (q+1) \left(1 - \left(\frac{q}{q+1}\right)^{t-j}\right),
\end{equation}
holds for every $0 \leqslant j \leqslant t$. This is indeed a more general result as, in particular, it follows from~\eqref{eq4} that 
$$
|{\mathcal F}_{i_0}| \geqslant \left(\frac{q}{q+1}\right)^t \cdot \frac{2 (q+1)^{t+1}}{q^t} - (q+1) + \frac{q^t}{(q+1)^{t-1}}\\
= (q+1) \left(\left(\frac{q}{q+1}\right)^t + 1\right) \geqslant 1,
$$
where the first inequality follows from our assumption that $|{\mathcal F}| \geqslant 2(q+1)^{t+1}/q^t$.

We prove~\eqref{eq4} by reverse induction on $j$. The base case $j = t$ holds trivially. Assume that~\eqref{eq4} holds for some $1 \leqslant j \leqslant t$; we prove it holds for $j-1$ as well. It follows by Waiter's strategy that $i_{j-1} \leqslant i_j + \lceil |{\mathcal F}_{i_j}|/(q+1) \rceil$. Since, moreover, Client claims exactly one offered element per round, we conclude that    
\begin{align*}
|{\mathcal F}_{i_{j-1}}| &\geqslant |{\mathcal F}_{i_j}| - \left\lceil\frac{|{\mathcal F}_{i_j}|}{q+1}\right\rceil \geqslant \frac{q}{q+1} |{\mathcal F}_{i_j}| - 1 \\
&\geqslant \frac{q}{q+1} \left[\left(\frac{q}{q+1}\right)^{t-j} |{\mathcal F}| - (q+1) \left(1 - \left(\frac{q}{q+1}\right)^{t-j}\right)\right] - 1\\
&= \left(\frac{q}{q+1}\right)^{t-j+1} |{\mathcal F}| - (q+1) \left(1 - \left(\frac{q}{q+1}\right)^{t-j+1}\right).
\end{align*}
{\hfill $\Box$ \medskip\\}

\section{Properties of Random Graphs} \label{sec::Gnp}

In this section we will prove several technical results about the binomial random graph $\mathcal{G}(n,p)$ for various edge probabilities $p$. These results will be useful in the following three sections and, in particular, in the proofs of Theorems~\ref{WCHamilton} and~\ref{CWHamilton}.

\begin{lemma} \label{NumberOfEdgesInSmallSets}
Let $G \sim \mathcal{G}(n,p)$, where $p = c \log n/n$ for some constant $c > 0$ and let $t = t(n)$ be such that $\lim_{n \rightarrow \infty} t \log n = \infty$. Then a.a.s. we have $e_G(A) \leqslant 2 c t |A| \log n$ for every $A \subseteq V(G)$ of size $1 \leqslant |A| \leqslant t n$.     
\end{lemma}

\begin{proof}
\begin{align*}
& \mathbb{P}\left[\exists A \subseteq V(G) \text{ such that } 1 \leqslant |A| \leqslant t n \text{ and } e_G(A) \geqslant 2 c t |A| \log n \right] \\
&\leqslant \sum_{a=1}^{t n} \binom{n}{a} \binom{\binom{a}{2}}{2 c t a \log n} p^{2 c t a \log n} \leqslant \sum_{a=1}^{t n} \left(\frac{en}{a}\right)^a \left(\frac{e \binom{a}{2} p}{2 c t a \log n}\right)^{2 c t a \log n}
\leqslant \sum_{a=1}^{t n} \left[\frac{en}{a} \cdot \left(\frac{e a}{4 t n}\right)^{2 c t \log n}\right]^a \\
&= \sum_{a=1}^{t n} \left[\exp\left\{1 + \log\left(n/a \right) + 2 c t \log n \left(1 - \log\left(n/a \right) - \log (4 t)\right) \right\} \right]^a = o(1).
\end{align*}
\end{proof}

\begin{lemma} \label{EdgesBetweenBigSets}
Let $G \sim \mathcal{G}(n,p)$ and let $k = k(n)$ be an integer satisfying $k p \geqslant 100 \log (n/k)$. Then a.a.s. $e_G(X, Y) \geqslant k^2 p/2$ holds for any pair of disjoint sets $X, Y \subseteq V(G)$ of size $|X| = |Y| = k$.
\end{lemma}

\begin{proof}
Let $X, Y \subseteq V(G)$ be arbitrary disjoint sets of size $|X| = |Y| = k$. Then $e_G(X, Y) \sim \text{Bin}(k^2, p)$ and thus
$$
\mathbb{P}\left[e_G(X,Y) < k^2 p/2 \right] = \mathbb{P}\left[e_G(X,Y) < \mathbb{E}(e_G(X,Y))/2 \right] < e^{- k^2 p/8},
$$
where the last inequality holds by Theorem~\ref{Chernoff}(i). 

A union bound over all choices of $X$ and $Y$ of size $k$ then gives
\begin{align*}
& \mathbb{P}\left[\exists X, Y \subseteq V(G) \text{ such that } |X| = |Y| = k, \, X \cap Y = \emptyset, \text{ and } e_G(X,Y) < k^2 p/2 \right]\\
&\leqslant \binom{n}{k}^2 \cdot e^{- k^2 p/8} \leqslant \left[\left(\frac{e n}{k} \right)^2 \cdot e^{- k p/8} \right]^k = \left[ \exp \{2 + 2 \log (n/k) - k p/8\} \right]^k = o(1), 
\end{align*}
where the last equality holds by our assumption on $k$.
\end{proof}

\begin{lemma} \label{sizeofR}
Let $c > 0$ be a constant and let $G \sim \mathcal{G}(n,p)$, where $p = c/n$. Then, a.a.s. $e(G) \leqslant c n$.
\end{lemma}

\begin{proof}
Clearly $e(G) \sim \text{Bin}(\binom{n}{2}, p)$. Hence
\begin{align*}
\mathbb{P}[e(G) > c n] \leqslant \mathbb{P}\left[e(G) > 1.5 \binom{n}{2} p\right] < \exp\left\{- \frac{\binom{n}{2} p}{12}\right\}
\leqslant \exp\left\{- c n/25\right\} = o(1),
\end{align*}
where the second inequality holds by Theorem~\ref{Chernoff}(ii).
\end{proof}

As noted in the introduction, an important part of proving Client's side in Theorem~\ref{CWHamilton}, is to show that a.a.s. the sum $\sum_{v \in V(G)} \left(\frac{q}{q+1} \right)^{d_G(v)}$ is very small, where $G \sim {\mathcal G}(n,p)$. The following lemma will play a key role in this endeavour. 

\begin{lemma} \label{UseOfBinomial}
Let $q$ be a positive integer and let $G \sim \mathcal{G}(n,p)$. For every $0 \leqslant i \leqslant n-1$, let $X_i = |\{u \in V(G) : d_G(u) = i\}|$ and let $\mu_i = \mathbb{E}[X_i]$. Then,
$$
\sum_{i=0}^{n-1} \left(\frac{q}{q+1}\right)^i \mu_i = n \left(1-\frac{p}{q+1}\right)^{n-1}.
$$
\end{lemma}

\begin{proof}
Let $\tilde{G} \sim \mathcal{G} \left(n,\frac{p}{q+1}\right)$ and let $Y$ denote the number of isolated vertices in $\tilde{G}$. Then,
\begin{equation} \label{eq2}
\mathbb{E}[Y] = n \left(1-\frac{p}{q+1}\right)^{n-1}.
\end{equation}

An alternative way of generating $\tilde{G}$ is by first generating $G \sim \mathcal{G}(n,p)$ and then deleting each edge of $G$ with probability $\frac{q}{q+1}$, independently of all other edges. It is then apparent that, for any $v \in V(G)$ with $d_G(v)=i$, we have
$$
\mathbb{P}[d_{\tilde{G}}(v)=0] = \left(\frac{q}{q+1}\right)^i.
$$
Hence,
\begin{equation} \label{eq3}
\mathbb{E}[Y] = \sum_{i=0}^{n-1} \left(\frac{q}{q+1}\right)^i \mu_i.
\end{equation}
Combining~\eqref{eq2} and~\eqref{eq3} we conclude that
$$
\sum_{i=0}^{n-1} \left(\frac{q}{q+1}\right)^i \mu_i = n \left(1-\frac{p}{q+1}\right)^{n-1}.
$$
\end{proof}

\begin{lemma} \label{UsingChebyshev}
Let $\varepsilon > 0$ be a constant, let $q$ be a positive integer and let $G \sim \mathcal{G}(n,p)$, where $p = (q + 1 - \varepsilon) \log n/n$. For every $0 \leqslant i \leqslant n-1$, let $X_i = |\{u \in V(G) : d_G(u) = i\}|$ and let $\mu_i = \mathbb{E}[X_i]$. If $0 \leqslant k \leqslant 9(q+1-\varepsilon) \log n$ is an integer such that $\mu_k \rightarrow \infty$, then a.a.s. $X_k \geqslant \mu_k/2$. 
\end{lemma}

\begin{proof}
Since
$$
\mathbb{P}[X_k < \mu_k/2] \leqslant \mathbb{P}[|X_k - \mu_k| \geqslant \mu_k/2] \leqslant \frac{4\text{Var}[X_k]}{\mu_k^2},
$$
where the last inequality holds by Chebyshev's inequality (Theorem~\ref{Chebyshev}), it suffices to show that $\text{Var}[X_k]/\mu_k^2 = o(1)$.

Let $v_1, \ldots, v_n$ denote the vertices of $G$. For every $1 \leqslant i \leqslant n$, let $Y_i$ be the indicator random variable taking the value $1$ if $d_G(v_i) = k$ and $0$ otherwise. Then
$$
\mathbb{E}[Y_i] = \mathbb{P}[Y_i = 1] = \binom{n-1}{k} p^k (1-p)^{n-1-k}.
$$

Moreover, $X_k = \sum_{i=1}^n Y_i$ and thus 
$$
\mu_k = \sum_{i=1}^n \mathbb{E}[Y_i] = n \binom{n-1}{k} p^k (1-p)^{n-1-k}.
$$

For every $1 \leqslant i \leqslant n$ we have
$$
\text{Var}[Y_i] = \mathbb{E}[Y_i^2] - (\mathbb{E}[Y_i])^2 = \mathbb{E}[Y_i] - (\mathbb{E}[Y_i])^2 \leqslant \mathbb{E}[Y_i],
$$
where the second equality holds since $Y_i^2 = Y_i$. Hence
$$
\sum_{i=1}^n \text{Var}[Y_i] \leqslant \sum_{i=1}^n \mathbb{E}[Y_i] = \mu_k.
$$

Fix some $1 \leqslant i \neq j \leqslant n$. Then 
\begin{align*}
\mathbb{E}[Y_i Y_j] &= \mathbb{P}[Y_i Y_j = 1] = \mathbb{P}[(Y_i = 1) \wedge (Y_j = 1)]\\
&= p \left[\binom{n-2}{k-1} p^{k-1} (1-p)^{n-1-k}\right]^2 + (1-p) \left[\binom{n-2}{k} p^k (1-p)^{n-2-k}\right]^2.
\end{align*}

Therefore,
\begin{align*}
\text{Cov}[Y_i, Y_j] &= \mathbb{E}[Y_i Y_j] - \mathbb{E}[Y_i] \mathbb{E}[Y_j]\\
&= p \left[\binom{n-2}{k-1} p^{k-1} (1-p)^{n-1-k}\right]^2 + (1-p) \left[\binom{n-2}{k} p^k (1-p)^{n-2-k}\right]^2\\
&\hspace*{2.25in} - \left[\binom{n-1}{k} p^k (1-p)^{n-1-k}\right]^2\\
&= \left[\binom{n-1}{k} p^k (1-p)^{n-1-k}\right]^2 \left[\left(\frac{k}{n-1}\right)^2 \frac{1}{p} + \left(1 - \frac{k}{n-1}\right)^2 \frac{1}{1-p} - 1\right].
\end{align*}

Hence,
\begin{align} \label{eq8}
\frac{1}{\mu_k^2} \sum_{1 \leqslant i \neq j \leqslant n} \text{Cov}[Y_i, Y_j] &= \frac{n(n-1)}{\mu_k^2} \left[\binom{n-1}{k} p^k (1-p)^{n-1-k}\right]^2 \left[\left(\frac{k}{n-1}\right)^2 \frac{1}{p} + \left(1 - \frac{k}{n-1}\right)^2 \frac{1}{1-p} - 1\right] \notag\\
&= \frac{n-1}{n} \left[\left(\frac{k}{n-1}\right)^2 \frac{1}{p} + \left(1 - \frac{k}{n-1}\right)^2 \frac{1}{1-p} - 1\right] \notag\\
&\leqslant \left(\frac{k}{n-1}\right)^2 \frac{1}{p} + \frac{p}{1-p} \notag\\
&\leqslant \left(\frac{9(q+1-\varepsilon) \log n}{n-1}\right)^2 \frac{n}{(q+1-\varepsilon) \log n} + \frac{(q+1-\varepsilon) \log n}{n - (q+1-\varepsilon) \log n}\notag\\
&\leqslant \frac{82 (q+1) \log n}{n-1} + \frac{2 (q+1) \log n}{n} = o(1).
\end{align}

Moreover, by our assumption on $k$ we have
\begin{equation} \label{eq7}
\frac{1}{\mu_k} = o(1).
\end{equation}

We conclude that
\begin{equation*}
\frac{\text{Var}[X_k]}{\mu_k^2} = \frac{1}{\mu_k^2} \left(\sum_{i=1}^n \text{Var}[Y_i] + \sum_{1 \leqslant i \neq j \leqslant n} \text{Cov}[Y_i, Y_j]\right) \leqslant \frac{1}{\mu_k} + \frac{1}{\mu_k^2} \sum_{1 \leqslant i \neq j \leqslant n} \text{Cov}[Y_i, Y_j] = o(1),
\end{equation*}
where the last equality holds by~\eqref{eq8} and~\eqref{eq7}.
\end{proof}

\begin{lemma}
\label{UpperBoundForUpperSumOfMu}
Let $G \sim \mathcal{G}(n,p)$, where $p = c \log n/n$ for some constant $c > \frac{2}{9\log 3}$. For every $0 \leqslant i \leqslant n-1$, let $X_i = |\{u \in V(G) : d_G(u) = i\}|$ and let $\mu_i = \mathbb{E}[X_i]$. Then 
$$
\sum_{i = 9c \log n}^{n-1} \mu_i = o(1).
$$
\end{lemma}

\begin{proof}
We first observe that the function $f(i) = (enp/i)^i$ is decreasing for $i \geqslant 9 c \log n$. Indeed, 
\begin{equation} \label{eq::decreasing}
\frac{f(i)}{f(i+1)} = \left(1+\frac{1}{i}\right)^i \cdot \frac{i+1}{enp} > \frac{i+1}{enp} \geqslant \frac{9c \log n}{e c \log n} = \frac{9}{e} > 1 \,.
\end{equation}

Then
\begin{align*}
\sum_{i = 9c \log n}^{n-1} \mu_i &= n \sum_{i = 9c \log n}^{n-1} \binom{n-1}{i} p^i (1-p)^{n-1-i} \leqslant n \sum_{i = 9c \log n}^{n-1} \left(\frac{enp}{i}\right)^i\\
&\leqslant n^2 \left(\frac{e}{9}\right)^{9c\log n} \leqslant \exp\left\{2\log n - 9c \log n \cdot \log 3 \right\} = o(1),
\end{align*}
where the second inequality holds by~\eqref{eq::decreasing} and the last equality follows from our choice of $c$.
\end{proof}

\begin{corollary} \label{UpperBoundOnMaxDegree}
Let $G \sim \mathcal{G}(n,p)$, where $p = c \log n/n$ for some constant $c > \frac{2}{9 \log 3}$. Then, a.a.s. $\Delta(G) \leqslant 9 c \log n$.
\end{corollary}

\begin{proof}
For every $0 \leqslant i \leqslant n-1$, let $X_i = |\{u \in V(G) : d_G(u) = i\}|$ and let $\mu_i = \mathbb{E}[X_i]$. Then
$$
\mathbb{P}[\Delta(G) \geqslant 9 c \log n] = \mathbb{P}[\exists i \text{ such that } 9c \log n \leqslant i \leqslant n-1 \text{ and } X_i > 0] \leqslant \sum_{i = 9c \log n}^{n-1} \mu_i = o(1), 
$$
where the first inequality follows from Markov's inequality and a union bound, and the last equality follows from Lemma~\ref{UpperBoundForUpperSumOfMu}.
\end{proof}

The following lemma is a fairly standard result in the theory of random graphs; for the sake of completeness we include its proof here.

\begin{lemma} \label{MinDegree}
Let $\varepsilon > 0$ be a constant and let $G \sim \mathcal{G}(n,p)$, where $p \geqslant (1 + \varepsilon) \log n/n$. Then there exists a constant $\gamma = \gamma(\varepsilon) > 0$ such that a.a.s. $\delta(G) \geqslant \gamma \log n$.   
\end{lemma}

\begin{proof}
By monotonicity, we can assume that $p = (1 + \varepsilon) \log n/n$. Let $0 < \gamma < 1$ be a constant satisfying $\gamma \log (e (1 + \varepsilon)/\gamma) < \varepsilon/3$; such a constant exists since $\lim_{\gamma \rightarrow 0} \gamma \log (1/\gamma) = 0$. We first observe that the function $f(i) = (enp/i)^i$ is increasing for $1 \leqslant i \leqslant \gamma \log n$. Indeed,
\begin{equation} \label{eq::increasing}
\frac{f(i)}{f(i+1)} = \left(1 + \frac{1}{i}\right)^i \cdot \frac{i+1}{enp} \leqslant \frac{i+1}{np} \leqslant \frac{\gamma \log n+1}{(1 + \varepsilon) \log n} < \gamma < 1 \,,
\end{equation}
where the last inequality holds by our choice of $\gamma$. 

Let $X$ be the random variable that counts the number of vertices of degree at most $\gamma \log n$ in $G$. Then,
\begin{align*}
\mathbb{E}[X] &= n \sum_{i=0}^{\gamma \log n} \binom{n-1}{i} p^i (1-p)^{n-1-i} \\
&\leqslant n \sum_{i=0}^{\gamma\log n}\binom{n}{i}p^i\exp\{-p(n-1-i)\}\\
&\leqslant n \exp\{-p(n-1)\}+n\exp\{-p(n-2\gamma\log n)\}\sum_{i=1}^{\gamma\log n}\left(\frac{enp}{i}\right)^i\\
&\leqslant n \exp\left\{-\left(1 + \varepsilon/2\right) \log n \right\} + n \exp\left\{-\left(1 + \varepsilon/2\right) \log n \right\} \cdot \gamma \log n \left(\frac{e (1 + \varepsilon) \log n}{\gamma \log n}\right)^{\gamma \log n}\\
&\leqslant n^{-\varepsilon/2} \left(1 + \exp\left\{\log \gamma + \log \log n + \gamma \log n \log \left(\frac{e(1 + \varepsilon)}{\gamma}\right)\right\} \right) \\
&= o(1),
\end{align*}
where the third inequality holds by~\eqref{eq::increasing} and the last equality follows from our choice of $\gamma$. Using Markov's inequality we conclude that
$$
\mathbb{P}[\delta(G) \leqslant \gamma \log n] = \mathbb{P}[X>0] \leqslant \mathbb{E}[X] = o(1).
$$
\end{proof}

\begin{lemma} \label{BoundBinomialCoefficient}
Let $r > 0$ be a constant and let $G \sim \mathcal{G}(n,p)$, where $p = c \log n/n$ for some constant $c > \frac{2}{9 \log 3}$. Then a.a.s. 
$$
\binom{d_G(v)}{r \log n} \leqslant n^{r(1 + \log (9c) + \log (1/r))},
$$
holds for any vertex $v \in V(G)$.
\end{lemma}

\begin{proof}
Since, by Corollary~\ref{UpperBoundOnMaxDegree}, a.a.s. $\Delta(G) \leqslant 9 c \log n$, it follows that a.a.s. 
\begin{align*}
\binom{d_G(v)}{r \log n} &\leqslant \left(\frac{e \cdot d_G(v)}{r \log n}\right)^{r \log n} \leqslant \left(\frac{e \cdot 9c}{r}\right)^{r \log n} = \exp\left\{r \log n \left(1 + \log (9c) + \log \left(1/r \right)\right)\right\}\\
&= n^{r(1 + \log (9c) + \log (1/r))} \,.
\end{align*}
\end{proof}

\begin{lemma} \label{BoundingTheSum}
Let $\varepsilon > 0$ be a constant, let $q$ be a positive integer and let $G \sim \mathcal{G}(n,p)$, where $p = (q + 1 + \varepsilon) \log n/n$. Then a.a.s.
$$
\sum_{v \in V(G)} \left(\frac{q}{q+1}\right)^{d_G(v)} \leqslant n^{- \varepsilon/(4(q+1))}.
$$
\end{lemma}

\begin{proof}
For every $0 \leqslant i \leqslant n-1$, let $X_i = |\{u \in V(G) : d_G(u) = i\}|$ and let $\mu_i = \mathbb{E}[X_i]$. Setting
$$
X = \sum_{i=0}^{n-1} \left(\frac{q}{q+1}\right)^i X_i = \sum_{v \in V(G)} \left(\frac{q}{q+1}\right)^{d_G(v)},
$$
it suffices to prove that a.a.s. $X \leqslant n^{-\varepsilon/(4(q+1))}$. Indeed, we have
\begin{align*}
\mathbb{E}[X] &= \sum_{i=0}^{n-1} \left(\frac{q}{q+1}\right)^i \mu_i = n \left(1-\frac{p}{q+1}\right)^{n-1} \leqslant n \exp\left\{- \frac{(q + 1 + \varepsilon) \log n}{(q+1) n} \cdot (n-1) \right\}\\
&\leqslant n \exp\left\{- \left(1 + \frac{\varepsilon}{2(q+1)}\right) \log n \right\} = n^{-\varepsilon/(2(q+1))},
\end{align*}
where the second equality follows from Lemma \ref{UseOfBinomial}. Therefore,
\begin{align*}
\mathbb{P}\left[X \geqslant n^{-\varepsilon/(4(q+1))}\right] &\leqslant n^{\varepsilon/(4(q+1))} \cdot \mathbb{E}[X] \leqslant n^{\varepsilon/(4(q+1)) - \varepsilon/(2(q+1))} = n^{-\varepsilon/(4(q+1))},
\end{align*}
where the first inequality follows from Markov's inequality.
\end{proof}

\begin{lemma} \label{P1}
Let $\varepsilon > 0$ be a constant, let $q$ be a positive integer and let $G \sim \mathcal{G}(n,p)$, where $p = (q + 1 + \varepsilon) \log n/n$. Then there exists a constant $r > 0$ such that the following holds. For every $v \in V(G)$, let $E(v) = \{e \in E(G) : v \in e\}$, let $\mathcal{A}(v) = \{A(v) \subseteq E(v) : |A(v)| = d_G(v) - r \log n\}$ and let $\mathcal{F}_1 = \bigcup_{v \in V(G)} \mathcal{A}(v)$. Then 
$$
\sum_{A \in \mathcal{F}_1} \left(\frac{q}{q+1}\right)^{|A|} = o(1).
$$
\end{lemma}

\begin{proof}
By Lemma~\ref{MinDegree} there exists a constant $\gamma > 0$ such that $\delta(G) \geqslant \gamma \log n$. Let $0 < r < \gamma$ be a constant satisfying  
$$
r \left(1 + \log(9(q+1+\varepsilon)) + \log\left(1/r \right) + \log\left(\frac{q+1}{q}\right)\right) < \frac{\varepsilon}{4(q+1)} \,.
$$
Such a constant $r$ exists since $\lim_{r \rightarrow 0} r \log (1/r) = 0$. Using this $r$ in the definition of $\mathcal{F}_1$, we obtain 
\begin{align*}
\sum_{A \in \mathcal{F}_1} \left(\frac{q}{q+1}\right)^{|A|} &= \sum_{v \in V(G)} \binom{d_G(v)}{r \log n} \left(\frac{q}{q+1}\right)^{d_G(v) - r \log n}\\
&\leqslant \left(\frac{q+1}{q}\right)^{r \log n} \cdot n^{r(1 + \log(9(q+1+\varepsilon)) + \log(1/r))} \cdot \sum_{v \in V(G)} \left(\frac{q}{q+1}\right)^{d_G(v)}\\
&\leqslant \exp\left\{r \log n \cdot \log\left(\frac{q+1}{q}\right)\right\} \cdot n^{r(1 + \log(9(q+1+\varepsilon)) + \log(1/r)) - \varepsilon/(4(q+1))}\\
&= n^{r \left(1 + \log(9(q+1+\varepsilon)) + \log \left(1/r \right) + \log\left(\frac{q+1}{q}\right)\right) - \varepsilon/(4(q+1))} = o(1),
\end{align*}
where the first inequality follows from Lemma \ref{BoundBinomialCoefficient}, the second inequality follows from Lemma~\ref{BoundingTheSum} and the last equality follows from our choice of $r$.
\end{proof}

\begin{lemma} \label{P2}
Let $\varepsilon > 0$ be a constant, let $q$ be a positive integer and let $G \sim \mathcal{G}(n,p)$, where $p = (q + 1 + \varepsilon) \log n/n$. Then there exists a constant $\lambda > 0$ for which
$$
\sum_{A \in \mathcal{F}_2} \left(\frac{q}{q+1}\right)^{|A|} = o(1),
$$
where $\mathcal{F}_2 = \left\{E_G(X,Y) : X,Y \subseteq V(G), \, |X| = |Y| = \frac{\lambda n\log\log n}{\log n} \text{ and } X \cap Y = \emptyset\right\}$.
\end{lemma}

\begin{proof}
Let $\lambda \geqslant 100$ be a constant satisfying $\lambda \log\left(\frac{q+1}{q}\right) > 2$. Then
\begin{align*}
\sum_{A \in \mathcal{F}_2} \left(\frac{q}{q+1}\right)^{|A|} &\leqslant \binom{n}{\frac{\lambda n \log \log n}{\log n}}^2 \left(\frac{q}{q+1}\right)^{\frac{\lambda^2 n(\log \log n)^2}{\log n}} \leqslant \left[\left(\frac{e\log n}{\lambda \log \log n}\right)^2 \left(\frac{q}{q+1}\right)^{\lambda \log \log n}\right]^{\frac{\lambda n \log \log n}{\log n}}\\
&\leqslant \left[\exp\left\{2 \log \log n - \lambda \log \log n \log\left(\frac{q+1}{q}\right) \right\}\right]^{\frac{\lambda n \log \log n}{\log n}} = o(1),
\end{align*}
where the first inequality follows since $q \geqslant 1$ and by Lemma~\ref{EdgesBetweenBigSets} which is applicable since $\lambda \geqslant 100$, and the last equality follows since $\lambda \log\left(\frac{q+1}{q}\right) > 2$ by assumption.
\end{proof}

\section{Expanders and Hamiltonicity} \label{sec::expanders}

In this section, we discuss the well-known relation between expanders and Hamiltonicity. We will make use of this relation in the following two sections where we will prove Theorems~\ref{WCHamilton} and~\ref{CWHamilton}.

\begin{definition}[Expander]
Let $G = (V,E)$ be a graph on $n$ vertices and let $t = t(n)$ and $k = k(n)$. The graph $G$ is called a \emph{$(t, k)$-expander} if $|N_G(U)| \geqslant k |U|$ for every set $U \subseteq V$ of size at most $t$.
\end{definition} 

The following result asserts that typically, for subgraphs of a random graph, large minimum degree is enough to ensure expansion.  

\begin{lemma} \label{MakingAnExpander}
Let $G \sim \mathcal{G}(n,p)$, where $p = c \log n/n$ for some constant $c > 0$, and let $\alpha = \alpha(n)$ and $k = k(n)$ be such that $\lim_{n \rightarrow \infty} \alpha k \log n = \infty$. Then a.a.s. every spanning subgraph $G' \subseteq G$ with minimum degree $\delta(G') \geqslant r \log n$ for some constant $r \geqslant 4 c \alpha (k + 1)^2>0$ is an $\left(\alpha n, k \right)$-expander.
\end{lemma}

\begin{proof}
Suppose for a contradiction that there exists a set $A \subseteq V(G)$ of size $1 \leqslant |A| \leqslant \alpha n$ and a spanning subgraph $G'\subseteq G$ such that $|N_{G'}(A)| < k |A|$. Then, $|A \cup N_{G'}(A)| < (k + 1) |A| \leqslant (k + 1) \alpha n$. It thus follows by Lemma~\ref{NumberOfEdgesInSmallSets} that a.a.s. 
\begin{equation} \label{eq::fewEdges}
e_{G'}(A \cup N_{G'}(A)) \leqslant 2 c (k + 1) \alpha |A \cup N_{G'}(A)| \log n < 2 c (k + 1)^2 \alpha |A| \log n \leqslant r |A| \log n/2.
\end{equation} 

On the other hand, since $\delta(G') \geqslant r \log n$, we have
$$
e_{G'}(A \cup N_{G'}(A)) \geqslant r |A| \log n/2   
$$
which clearly contradicts~\eqref{eq::fewEdges}. We conclude that $G'$ is indeed an $\left(\alpha n, k \right)$-expander. 
\end{proof}

\begin{definition}[Booster] \label{def::booster}
A non-edge $uv$ of a graph $G$, where $u, v \in V(G)$, is called a \emph{booster with respect to $G$} if $G \cup \{uv\}$ is Hamiltonian or its longest path is strictly longer than that of $G$. We denote the set of boosters with respect to $G$ by ${\mathcal B}_G$. 
\end{definition}

The following lemma (see, e.g.,~\cite{FK}), which is essentially due to P\'osa~\cite{Posa1976}, asserts that expanders have many boosters. 

\begin{lemma} \label{Posa}
If $G$ is a connected non-Hamiltonian $(t, 2)$-expander, then $|{\mathcal B}_G| \geqslant (t+1)^2/2$. 
\end{lemma}

Using P\'osa's lemma, we will show that, if a sparse subgraph of a random graph is a good expander, then it has many boosters in the random graph.

\begin{lemma} \label{Prop4}
Let $\varepsilon, s_1 < 1$ and $s_2 < 1/100$ be positive constants and let $G \sim \mathcal{G}(n,p)$, where $p = (1 + \varepsilon) \log n/n$. If $s_1 (1 - \log s_1) < 1/400$, then a.a.s. every connected non-Hamiltonian $(n/5, 2)$-expander $\Gamma \subseteq G$ with at most $s_1 n \log n$ edges has at least $s_2 n \log n$ boosters in $G$.
\end{lemma}

\begin{proof}
For a connected non-Hamiltonian $(n/5, 2)$-expander $\Gamma \subseteq G$ with at most $s_1 n \log n$ edges, let $X_\Gamma = |\mathcal{B}_\Gamma \cap E(G)|$. Then $X_\Gamma \sim \text{Bin}(|\mathcal{B}_\Gamma|, p)$ and, by Lemma~\ref{Posa}, $|\mathcal{B}_\Gamma| \geqslant n^2/50$. Therefore,
\begin{align*}
\mathbb{P}[X_\Gamma < s_2 n \log n] < \exp\left\{- \left(1 - \frac{50s_2}{1 + \varepsilon}\right)^2 n^2 p/100\right\} \leqslant \exp\left\{- n \log n/400\right\}, 
\end{align*}
where the first inequality follows from Theorem~\ref{Chernoff}(i) with $a = 1 - 50 s_2/(1 + \varepsilon)$ and the last inequality holds since $s_2 \leqslant 1/100$ and $\varepsilon > 0$.

Taking a union bound over all spanning subgraphs of $G$ which are connected non-Hamiltonian $(n/5, 2)$-expanders with at most $s_1 n \log n$ edges, we conclude that the probability that there exists such a subgraph with less than $s_2 n \log n$ boosters in $G$ is at most
\begin{align*}
& \sum_{m=1}^{s_1 n \log n} \binom{\binom{n}{2}}{m} p^m \cdot \exp\left\{- n \log n/400\right\} \leqslant \exp\left\{- n \log n/400\right\} \cdot \sum_{m=1}^{s_1 n \log n} \left(\frac{e n \log n}{m} \right)^m \\  
&\leqslant \exp\left\{- n \log n/400\right\} \cdot s_1 n \log n \cdot \left(\frac{e}{s_1} \right)^{s_1 n \log n} \\ 
&\leqslant \exp\left\{2 \log n + s_1 n \log n (1 - \log s_1) - n \log n/400\right\} = o(1),  
\end{align*}
where the first inequality holds since $\varepsilon < 1$, the second inequality holds since $f(m) = (e n \log n/m)^m$ is increasing for $1 \leqslant m \leqslant s_1 n \log n$ as can be shown by a calculation similar to~\eqref{eq::increasing}, and the last equality holds since $s_1 (1 - \log s_1) < 1/400$ by assumption. 
\end{proof}

We end this section by recalling a sufficient condition for Hamiltonicity from~\cite{Hefetz2009b}; it is based on expansion and high connectivity. 

\begin{theorem} [\cite{Hefetz2009b}] \label{HamCriteria}
Let $12 \leqslant d \leqslant e^{\sqrt[3]{\log n}}$ and let $G$ be a graph on $n$ vertices which satisfies the following two properties.
\begin{description}
\item[P1] For every $S \subseteq V(G)$, if $|S| \leqslant \frac{n \log \log n \log d}{d \log n \log \log \log n}$, then $|N_G(S)| \geqslant d|S|$;
\item[P2] There exists an edge in $G$ between any two disjoint subsets $A, B \subseteq V(G)$ of size $|A|, |B| \geqslant \frac{n \log \log n \log d}{4130 \log n \log \log \log n}$.
\end{description}
Then $G$ is Hamiltonian for sufficiently large $n$.
\end{theorem}

\section{The Waiter-Client Hamiltonicity game} \label{sec::mainWC}

\noindent \emph{Proof of Theorem~\ref{WCHamilton}}. Let $\varepsilon > 0$ be a constant. For $p = (1 - \varepsilon) \log n/n$, it is well-known (see, e.g.,~\cite{BollobasBook, Janson2000, FriezeKaronski}) that a.a.s. $G \sim \mathcal{G}(n,p)$ has an isolated vertex and therefore is not Hamiltonian. Hence, a.a.s. Client wins the $(1 : q)$ Waiter-Client Hamiltonicity game on $E(G)$ regardless of his strategy.    

Assume then that $G \sim \mathcal{G}(n,p)$, where $p = (1 + \varepsilon) \log n/n$ for some constant $\varepsilon > 0$. We present a strategy for Waiter to win the $(1 : q)$ Waiter-Client Hamiltonicity game on $E(G)$ and then prove that a.a.s. he can play according to this strategy. Waiter's strategy consists of the following four stages (the constants $\bar{c}, c_1$ and $c_2$ appearing in the description of the strategy will be determined later).

\paragraph{Stage 0:} Waiter splits $G$ into two spanning subgraphs, the main graph $G_M$ and a reservoir graph $R$, by placing each edge of $G$ in $R$ independently with probability $\bar{p} = \bar{c}/\log n$, and then setting $E(G_M) = E(G) \setminus E(R)$.

\paragraph{Stage 1:} By only offering edges from $G_M$ and following the strategy given by Lemma 3.4, Waiter forces Client to build a $(c_1 n, 2)$-expander $G_1$ with at most $c_2 n \log n$ edges.

\paragraph{Stage 2:} By only offering the edges of $R$ and following the strategy given by Theorem 3.1, Waiter forces Client to build a graph $G_2$ such that $G_1 \cup G_2$ is an $(n/5, 2)$-expander.

\paragraph{Stage 3:} For as long as $G_C$ is not Hamiltonian, in each round Waiter offers Client $q+1$ free boosters with respect to $G_C$. Once $G_C$ becomes Hamiltonian, Waiter plays arbitrarily for the remainder of the game.

\vspace{0.1in}

It is evident from the description of Stage $3$ of the proposed strategy that, if Waiter is able to play according to this strategy, then he wins the game. Moreover, it is clear that Waiter can follow Stage $0$ of the strategy. It thus remains to prove that he can follow Stages 1--3 as well. We consider each stage in turn.

\paragraph{Stage 1:} We first observe that 
$$
p (1 - \bar{p}) = (1 + \varepsilon)(\log n - \bar{c})/n \geqslant (1 + \varepsilon/2) \log n/n ,
$$ 
and that $G_M \sim \mathcal{G}(n, p (1 - \bar{p}))$. It then follows from Lemma~\ref{MinDegree} that a.a.s. $\delta(G_M) \geqslant \gamma \log n$ for some constant $\gamma > 0$. Let $0 < c_2 < 1/(600(q+1))$ be a constant satisfying $\lfloor \gamma \log n/(2(q+1)) \rfloor \geqslant c_2 \log n$ and $3 c_2 (1 - \log (3 c_2)) < 1/400$. By Lemma~\ref{lem::largeMinDeg}, Waiter has a strategy to force Client to build a spanning subgraph $G_1$ of $G_M$ with minimum degree $\delta(G_1) \geqslant c_2 \log n$, by offering at most $(q + 1) c_2 n \log n$ edges of $G_M$; in particular, $e(G_1) \leqslant c_2 n \log n$. Finally, it follows by Lemma~\ref{MakingAnExpander} that $G_1$ is a $(c_1 n, 2)$-expander, for a sufficiently small constant $c_1 > 0$.     

\paragraph{Stage 2:} Let $\mathcal{F} = \{E_R(X, Y) : X, Y \subseteq V(G), \, |X| = |Y| = c_1 n \text{ and } X \cap Y = \emptyset\}$. Since $R \sim \mathcal{G}(n, p \bar{p})$ and $p \bar{p} = (1 + \varepsilon) \bar{c}/n$, we have
\begin{align*}
\sum_{A \in \mathcal{F}} 2^{- |A|/(2q-1)} &\leqslant \binom{n}{c_1 n}^2 2^{- 0.5 c_1^2 \bar{c}(1 + \varepsilon)n/(2q-1)}\\
&\leqslant \left(\frac{e}{c_1}\right)^{2 c_1 n} 2^{- c_1^2 \bar{c} n/(4q)}\\
&= \exp\left\{2 c_1 n \left(1 - \log c_1 \right) - \frac{c_1^2 \bar{c} n \log 2}{4q}\right\}\\
& = o(1),
\end{align*}
where the first inequality follows from Lemma~\ref{EdgesBetweenBigSets} which is applicable for a sufficiently large constant $\bar{c}$, and the last equality holds for sufficiently large $\bar{c}$. Hence, by Theorem~\ref{Bednarska}, and since all edges of $R$ are free at the beginning of Stage $2$, Waiter has a strategy to force Client to claim an edge of $R$ between every pair of disjoint sets of vertices of $G$, each of size $c_1 n$.

Let $G_2$ denote the graph built by Client in Stage $2$. We claim that $G_1 \cup G_2$ is an $(n/5, 2)$-expander. Since $G_1$ is a $(c_1 n, 2)$-expander and expansion is a monotone increasing property, it suffices to demonstrate expansion for sets $A \subseteq V(G)$ of size $c_1 n \leqslant |A| \leqslant n/5$. Suppose for a contradiction that $A \subseteq V(G)$ is a set of size $c_1 n \leqslant |A| \leqslant n/5$ and yet $|N_{G_1 \cup G_2}(A)| < 2 |A|$. Then $|V(G) \setminus (A \cup N_{G_1 \cup G_2}(A))| > n - 3 |A| \geqslant 2n/5 \geqslant c_1 n$ and there are no edges of $G_1 \cup G_2$ between $A$ and $V(G) \setminus (A \cup N_{G_1 \cup G_2}(A))$. This contradicts the way $G_2$ was constructed. We conclude that $G_1 \cup G_2$ is indeed an $(n/5, 2)$-expander at the end of Stage $2$.

\paragraph{Stage 3:} Observe that, at the end of Stage $2$, Client's graph $G_C$ is connected. Indeed, since $G_1 \cup G_2$ is an $(n/5, 2)$-expander, each of its connected components must have size at least $3n/5$ and thus there can be only one such component. It follows that, at the beginning of Stage $3$, Client's graph is a connected $(n/5, 2)$-expander. Since connectivity and expansion are monotone increasing properties, this remains true for the remainder of the game. We will show that this allows Waiter to offer Client $q+1$ free boosters in every round of Stage $3$ until $G_C$ becomes Hamiltonian. 

It is evident from Definition~\ref{def::booster} that one needs to sequentially add at most $n$ boosters to an $n$-vertex graph to make it Hamiltonian. Hence, in order to prove that Waiter can follow Stage $3$ of the proposed strategy, it suffices to show that, for every $1 \leqslant i \leqslant n$, if $G_C$ is not Hamiltonian at the beginning of the $i$th round of Stage $3$, then $|{\mathcal B}_{G_C} \cap E(G_F)| \geqslant q + 1$ holds at this point. By the description of Stage $1$ we have $e(G_1) \leqslant c_2 n \log n$ and by the description of Stage $2$ we have $e(G_2) \leqslant e(R) \leqslant (1 + \varepsilon) \bar{c} n$, where the last inequality holds a.a.s. by Lemma~\ref{sizeofR}. Hence, a.a.s. $e(G_1 \cup G_2) \leqslant 2 c_2 n \log n$. 

Fix an integer $1 \leqslant i \leqslant n$ and suppose that $G_C$ is not Hamiltonian at the beginning of the $i$th round of Stage $3$. Then $G_C$ is a connected, non-Hamiltonian $(n/5, 2)$-expander with at most $2 c_2 n \log n + (i-1) \leqslant 3 c_2 n \log n$ edges. Since, moreover, $c_2$ was chosen such that $3 c_2 (1 - \log (3 c_2)) < 1/400$, it follows by Lemma~\ref{Prop4} that $|{\mathcal B}_{G_C} \cap E(G)| \geqslant n \log n/200$. We conclude that
$$
|{\mathcal B}_{G_C} \cap E(G_F)| \geqslant |{\mathcal B}_{G_C} \cap E(G)| - (e(G_C) + e(G_W)) \geqslant n \log n/200 - 3 c_2 (q+1) n \log n \geqslant q+1,
$$ 
where the last inequality holds since $c_2 < 1/(600(q+1))$ by assumption.
{\hfill $\Box$ \medskip\\}

\section{The Client-Waiter Hamiltonicity game} \label{sec::mainCW}

\noindent \emph{Proof of Theorem~\ref{CWHamilton}}. Assume first that $G \sim {\mathcal G}(n,p)$, where $p = (q + 1 + \varepsilon) \log n/n$ for some positive constant $\varepsilon$. We will present a strategy for Client for the $(1 : q)$ Client-Waiter Hamiltonicity game on $E(G)$; it is based on the sufficient condition for Hamiltonicity from Theorem~\ref{HamCriteria}. Let $r$ and $\mathcal{F}_1$ be as in Lemma~\ref{P1} and let $\lambda$ and $\mathcal{F}_2$ be as in Lemma~\ref{P2}. Note that $\sum_{A \in \mathcal{F}_1} \left(\frac{q}{q+1}\right)^{|A|} = o(1)$ holds by Lemma~\ref{P1} and that $\sum_{A \in \mathcal{F}_2} \left(\frac{q}{q+1}\right)^{|A|} = o(1)$ holds by Lemma~\ref{P2}. Let $\mathcal{F} = \mathcal{F}_1 \cup \mathcal{F}_2$. Then
$$
\sum_{A \in \mathcal{F}} \left(\frac{q}{q+1}\right)^{|A|} = \sum_{A \in \mathcal{F}_1} \left(\frac{q}{q+1}\right)^{|A|} + \sum_{A \in \mathcal{F}_2} \left(\frac{q}{q+1}\right)^{|A|} = o(1).
$$    
It thus follows by Theorem~\ref{ClientWinningTransversalGame} that Client has a winning strategy for the $(1 : q)$ Client-Waiter game $(E(G), \mathcal{F}^*)$. 

We claim that if Client follows this strategy, then his graph at the end of the game satisfies properties {\bf P1} and {\bf P2} from Theorem~\ref{HamCriteria}, with $d = (\log n)^{1/3}$, and is therefore Hamiltonian. Indeed, it follows from the definition of $\mathcal{F}_1$ that, at the end of the game, the minimum degree in Client's graph will be at least $r \log n$. Using Lemma~\ref{MakingAnExpander}, it is then easy to verify that Client's graph is an $(n/\log n, (\log n)^{1/3})$-expander and thus satisfies property {\bf P1}. Moreover, a straightforward calculation shows that, by the definition of $\mathcal{F}_2$, at the end of the game, Client's graph will satisfy property {\bf P2} as well.   

\bigskip   
  
Next, assume that $G \sim {\mathcal G}(n,p)$, where $p = (q + 1 - \varepsilon) \log n/n$ for some positive constant $\varepsilon$. We will present a strategy for Waiter to isolate a vertex in Client's graph.

\vspace*{-4mm}

\paragraph{Waiter's strategy:} Let $k$ be a positive integer and let $I_k$ be an independent set in $G$ such that $|I_k| \geqslant 2 (q+1)^{k+1}/q^k$ and $d_G(u) = k$ for every $u \in I_k$. For every $u \in I_k$, let $E(u) = \{e \in E(G) : u \in e\}$ and let $X = \bigcup_{u \in I_k} E(u)$. Waiter isolates a vertex of $I_k$ in Client's graph by following the strategy for the $(1 : q)$ box game on $\{E(u) : u \in I_k\}$ which is described in the proof of Proposition~\ref{BoxGame}. 

Since $|I_k| \geqslant 2 (q+1)^{k+1}/q^k$, it follows by Proposition~\ref{BoxGame} that Waiter can indeed isolate a vertex in Client's graph. Hence, it remains to prove that Waiter can play according to the proposed strategy. In order to do so, it suffices to show that a.a.s. a positive integer $k$ and an independent set $I_k$ as above exist.  

For every $0 \leqslant i \leqslant n-1$, let $X_i = |\{u \in V(G) : d_G(u) = i\}|$ and let $\mu_i = \mathbb{E}[X_i]$. Then
\begin{align} \label{eq::largeSum}
\sum_{i=0}^{n-1} \left(\frac{q}{q+1}\right)^i \mu_i &= n \left(1 - \frac{p}{q+1}\right)^{n-1} \geqslant n \exp\left\{-\frac{n-1}{n} \left(\frac{(q+1-\varepsilon) \log n}{q+1} + \frac{(q+1-\varepsilon)^2 \log^2 n}{n(q+1)^2}\right)\right\} \nonumber\\
&\geqslant n \exp\left\{-\left(1 - \frac{\varepsilon}{2(q+1)}\right) \log n\right\} \geqslant n^{\delta} ,
\end{align}
where the first equality holds by Lemma~\ref{UseOfBinomial}, the first inequality follows from the fact that $e^{-(x+x^2)} \leqslant 1-x$ holds for sufficiently small $x > 0$ by the Taylor expansion of $e^{-y}$, and the last inequality holds for a sufficiently small constant $\delta > 0$. Since, moreover, 
$$
\sum_{i = 9(q+1-\varepsilon) \log n}^{n-1} \left(\frac{q}{q+1}\right)^i \mu_i = o(1),
$$
holds by Lemma~\ref{UpperBoundForUpperSumOfMu}, it follows from~\eqref{eq::largeSum} that
\begin{align*}
\sum_{i=0}^{9(q+1-\varepsilon)\log n} \left(\frac{q}{q+1}\right)^i \mu_i &\geqslant n^{\delta}/2.
\end{align*}

Hence, there exists an integer $0 \leqslant k \leqslant 9(q+1-\varepsilon) \log n$ such that
$$
\left(\frac{q}{q+1}\right)^k \mu_k \geqslant \frac{n^{\delta}}{18 (q+1) \log n}.
$$
In particular, $\mu_k \rightarrow \infty$ as $n \rightarrow \infty$ holds for this value of $k$ and thus, by Lemma~\ref{UsingChebyshev}, a.a.s $X_k \geqslant \mu_k/2$. It follows that a.a.s. 
\begin{equation} \label{eq::largeXk}
\frac{q^k}{2(q+1)^{k+1}} \cdot X_k \geqslant \frac{q^k}{2(q+1)^{k+1}} \cdot \frac{\mu_k}{2} \geqslant \frac{n^{\delta}}{72 (q+1)^2 \log n}.
\end{equation}

Let $S_k = \{u \in V(G) : d_G(u) = k\}$ and let $I_k \subseteq S_k$ be an independent set of maximum size. It is easy to see that
$$
|I_k| \geqslant \frac{|S_k|}{k+1} = \frac{X_k}{k+1} \geqslant \frac{n^{\delta}}{72 (k+1) (q+1)^2 \log n} \cdot \frac{2(q+1)^{k+1}}{q^k} \geqslant \frac{2(q+1)^{k+1}}{q^k},  
$$
where the second inequality holds by~\eqref{eq::largeXk} and the last inequality holds for sufficiently large $n$ since $k \leqslant 9 (q+1-\varepsilon) \log n$.   
{\hfill$\Box$\medskip\\}

\section{Concluding remarks and open problems}
\label{sec::openprob}

In this paper, we determined sharp thresholds for the $(1 : q)$ Waiter-Client and Client-Waiter Hamiltonicity games, played on the edge set of the random graph ${\mathcal G}(n,p)$, for every fixed $q$. For the Waiter-Client version, it is $\log n/n$; in particular it does not depend on $q$. This is asymptotically the same as the sharp threshold for the appearance of a Hamilton cycle in ${\mathcal G}(n,p)$. On the other hand, the sharp threshold for the Client-Waiter Hamiltonicity game on ${\mathcal G}(n,p)$ is $(q+1) \log n/n$ and thus does grow with $q$. It is natural to study the behaviour of these thresholds for non-constant values of $q$ as well. As noted in the introduction, for Maker-Breaker and Avoider-Enforcer games, this was done in~\cite{FGKN} for every value of $q$ for which Maker (respectively Enforcer) has a winning strategy for the $(1 : q)$ Maker-Breaker (respectively Avoider-Enforcer) Hamiltonicity game on $K_n$. It was proved in~\cite{BHKL} that the largest $q$ for which Waiter has a winning strategy in the $(1 : q)$ Waiter-Client Hamiltonicity game on $K_n$ is of linear order. Moreover, using a similar argument to the one employed in~\cite{KS}, it is not hard to show that the largest $q$ for which Client has a winning strategy in the $(1 : q)$ Client-Waiter Hamiltonicity game on $K_n$ is $(1 - o(1)) n/\log n$. It would be interesting to determine the threshold probabilities for the graph property ${\mathcal W}_{\mathcal H}^q$ for every $q = O(n)$ and for the graph property ${\mathcal C}_{\mathcal H}^q$ for every $q \leqslant (1 - o(1)) n/\log n$.

\section*{Acknowledgement}
We would like to thank the anonymous referees for their helpful comments.

\bibliographystyle{amsplain}

\end{document}